\documentclass[a4paper,11pt]{article}

\usepackage[affil-it]{authblk}

\usepackage{amsfonts}
\usepackage{amsmath}
\usepackage{amssymb}
\usepackage{amsthm}

\newtheorem{theorem}{Theorem}[section]
\newtheorem{proposition}[theorem]{Proposition}
\newtheorem{lemma}[theorem]{Lemma}
\newtheorem{corollary}[theorem]{Corollary}

\theoremstyle{definition}
\newtheorem{definition}[theorem]{Definition}
\theoremstyle{remark}
\newtheorem{remark}[theorem]{Remark}
\theoremstyle{remark}

\theoremstyle{remark}

\newcommand{\supp}{\mathrm{supp}}
\newcommand{\id}{\mathrm{Id}}

\begin{document}

\title{Finite Factor Representations of Higman-Thompson groups}

\author{Artem Dudko%
  \thanks{email: \texttt{artem.dudko@stonybrook.edu}}}
\affil{Department of Mathematics, Stony Brook University,   NY}

\author{Konstantin Medynets%
  \thanks{email: \texttt{medynets@usna.edu}}}
\affil{Department of Mathematics, U.S. Naval Academy, Annapolis, MD}

\date{}
\maketitle

\begin{abstract} We prove that  the only finite factor-representations of  the Higman-Thompson groups  $\{F_{n,r}\}$, $ \{G_{n,r}\}$ are the regular representations and  scalar representations arising from   group  abelianizations.  As a corollary, we obtain that any measure-preserving ergodic action  of a simple Higman-Thompson group must be essentially free. Finite factor representations of other classes of groups are also discussed.
\end{abstract}

\section{Introduction}

The goal of this paper is to describe finite (in the sense of Murray-von Neumann) factor-representations of the Higman-Thompson groups (see Section \ref{Applications} for the definition). The discussion of historical importance of these groups and their various algebraic properties can be found  in \cite{Brown:1987}, \cite{CFP:1996}, and \cite{Stein:1992}.  The following  is the main result of the present paper. \medskip

\noindent  {\bf Theorem. }{\it Let $G$ be a group from the Higman-Thompson families $\{F_{n,r}\}$, $ \{G_{n,r}\}$ and $\pi$ be a finite factor representation of $G$. Then either $\pi$ is the regular representation or $\pi$ has  the form $$\pi(g) = \rho([g])\id,$$ where $[g]$ is the image of $g$ in the abelianization $G/G'$, $G'$ is the commutator of $G$, $\rho: G/G' \rightarrow \mathbb T$ is a group homomorphism and $\id$ is the identical operator in some Hilbert space.}\medskip

 Since finite factor representations are in one-to-one correspondence with positive definite class functions (termed {\it characters}, see Definition \ref{DefinitionCharacter}),  this  shows that the characters of any Higman-Thompson group $G$   are convex combinations of the  regular character and  characters of its abelianization $G/G'$. The structure of group characters has implications on dynamical properties of group actions.  Suppose that a group $G$ admitting  no non-regular/non-identity characters acts on a probability measure space $(X,\mu)$ by measure-preserving transformations. Setting $\chi(g) = \mu(Fix(g))$, $Fix(g) = \{x\in X : g(x) = x\}$, one can check that the function $\chi$ is a character. So our results imply that  any non-trivial ergodic action of $G$  on a probability measure space $(X,\mu)$ is  essentially free, i.e.  $\mu(Fix(g)) = 0$ for every $g\in G\setminus \{e\}$, see Theorem \ref{TheoremErgodicAction}.

In our proofs, we mostly utilize the fact that the commutators of Higman-Thompson groups have no non-atomic invariant measures on the spaces where they are defined. This means that the orbit equivalence relations generated by their actions are compressible \cite{Kechris:1994}. This observation allows us to state the main result in terms of dynamical properties of group actions (Theorems \ref{TheoremNoCharactersSimpleGroups} and \ref{TheoremCharactersNonSimpleGroups}) ---   transformation groups whose actions are ``compressible''  (Definition \ref{DefinitionCompressibleBase}) do  not admit non-regular $II_1$-factor representations, except for possible finite-dimensional representations. This dynamical formulation allows us to apply the main result to other classes of transformation groups (Section \ref{Applications}).

In \cite{Vershik:NonfreeActions} Vershik suggested that the characters  of ``rich''  groups  should often arise as $\mu(Fix(g))$ for some invariant measure $\mu$. Thus, this paper confirms Vershik's conjecture  in the sense that  the absence of non-trivial invariant measures implies the absence of non-regular characters. We also mention the paper \cite{Dudko_Medynets:2012}, where Vershik's conjecture was established  for full groups of Bratteli diagrams.

The structure of the paper is the following. In Section \ref{SectionGeneralTheory} we build the general theory of finite factor representations for groups admitting compressible  actions. In Section \ref{Applications}, we apply our general results to the Higman-Thompson groups and to the full groups of irreducible shifts of finite type \cite{Matui:2012}.

%
%
%

\section{General Theory}\label{SectionGeneralTheory}

In this section we show that if a group admits a compressible action on a topological space, then this group, under some algebraic assumptions, has no non-trivial factor representations.  We will start with definitions  from the representation theory of infinite groups.

\begin{definition}\label{DefinitionCharacter}
A {\it character} of a group $G$ is a function $\chi:G\rightarrow
\mathbb{C}$ satisfying the following conditions
\begin{itemize}
\item[1)] $\chi(g_1g_2)=\chi(g_2g_1)$ for any $g_1,g_2\in G$;
\item[2)] the matrix
$\left\{\chi\left(g_ig_j^{-1}\right)\right\}_{i,j=1}^n$ is
nonnegatively defined for any $n$ and $g_1,\ldots,g_n\in G$;
\item[3)] $\chi(e)=1$. Here $e$ is the group identity.
\end{itemize}

A character $\chi$ is called {\it indecomposable} if it
cannot be written in the form $\chi=\alpha
\chi_1+(1-\alpha)\chi_2$, where $0<\alpha<1$ and $\chi_1,\chi_2$ are
distinct characters.
\end{definition}

For a unitary representation $\pi$ of a group $G$ denote by
$\mathcal{M}_\pi$ the $W^{*}$-algebra generated by the operators of
the representation $\pi$.
 Recall that the {\it commutant} $S'$ of a set
$S$ of operators in a Hilbert space $ H $ is the algebra
$S'=\{A\in B( H ):AB=BA\text{ for any }B\in S\}$.

\begin{definition} A representation $\pi$ of a group $G$ is called a {\it
factor representation} if the algebra $\mathcal{M}_\pi$ is a
factor, that is $\mathcal{M}_\pi \cap \mathcal{M}_\pi'=\mathbb{C}\id$.
\end{definition}

 The indecomposable characters on a group $G$ are in one-to-one correspondence with the {\it finite type}\footnote{The classification of factors  can be found in \cite[Chapter 5]{Tak}.} factor  representations
of $G$. Namely, starting with an indecomposable character $\chi$ on
$G$ one can construct a triple
$\left(\pi, H,\xi\right)$, referred to as the
{\it Gelfand-Naimark-Siegal} (abbr. GNS) {\it construction}. Here $\pi$ is a finite
type factor representation acting in the space $H$,
and $\xi$ is a unit vector in $ H$ such that
$\chi(g)=(\pi(g)\xi,\xi)$ for every $g\in G$, see, for example, \cite[Sect. 2.3]{Dudko_Medynets:2012}. Note
that the vector $\xi$ is cyclic and separating for the von Neumann  algebra $\mathcal{M}_\pi$. The latter means that if $A\xi = 0$ for some $A\in \mathcal M_\pi$, then $A = 0$.

\begin{remark} We note that each character defines a factor representation up to {\it quasi-equivalence}. Two unitary representations $\pi_1$ and $\pi_2$ of
the same group $G$ are called {\it quasi-equivalent} if there is an isomorphism of von
Neumann algebras $\omega :\mathcal M_{\pi_1} \rightarrow \mathcal M_{\pi_2}$ such that $\omega(\pi_1(g)) = \pi_2(g)$  for each $g\in G$. For example, all $II_1$ factor representations of an amenable group are hyperfinite \cite[Corollary 6.9 and Theorem 6]{Connes:1976} and, hence, generate isomorphic algebras. At the same time, they might be not quasi-equivalent.
\end{remark}

Suppose that $G$ is an  infinite conjugacy class (abbr. ICC) group. Then its left regular representation $\pi$ generates a $II_1$-factor and the function $\chi(g) = (\pi(g)\delta_e,\delta_e)$ is an indecomposable  character (termed the {\it regular character}).

\begin{definition}\label{DefinitionCompressible} We will say that a group $H$ {\it has no  proper characters} if $\chi$ being an indecomposable character of $H$ implies that either $\chi$ is {\it identity character} given by $$\chi(g) = 1\;\;\text{for every}\;\;g\in G$$ or the {\it regular character} defined as $$\chi(g) = 0\;\;\text{if}\;\;g\neq e\;\;\text{and}\;\;\chi(e) = 1.$$ We notice that for non-ICC groups the regular characters are decomposable. 
\end{definition}

Fix   a {\it regular Hausdorff topological space} $X$. Notice that any two distinct points of $X$ have open neighborhoods with disjoint closures.  To exclude trivial counterexamples to our statements we assume that the set $X$ is infinite.  Suppose that a group $G$ acts on $X$. For a group element $g\in G$, denote its {\it support} by $supp(g) = \overline{\{x\in X : g(x)\neq x\}}$.

\begin{definition}\label{DefinitionCompressibleBase}  We will say that the action of $G$ on $X$ is \emph{compressible}  if there exists a base of the topology $\mathfrak{U}$ on $X$ such that
\begin{itemize}\item[(i)] for every $g\in G$  there exists $U\in\mathfrak{U}\,$ such that $\supp(g)\subset U$;
\item[(ii)] for every $U_1,U_2\in \mathfrak{U}\,$ there exists $g\in G$ such that $g(U_1)\subset U_2$;
\item[(iii)] for every $U_1,U_2,U_3\in\mathfrak{U}$ with $\overline{U}_1\cap \overline{U}_2=\varnothing$ there exists $g\in G$ such that $g(U_1)\cap U_3=\varnothing$ and  $\supp(g)\cap U_2=\varnothing$.
\item[(iv)] for any $U_1,U_2\in \mathfrak{U}$ there exists $U_3\in\mathfrak{U}$ such that $U_3\supset U_1\cup U_2$.
\end{itemize}
\end{definition}

\begin{remark} Suppose that $X$ is a Polish space. If an action of  $G$ on $X$ is compressible, then  the $G$-action has no probability invariant  measure. The latter is equivalent to  the $G$-orbit equivalence relation being compressible (see \cite{Kechris:1994} and references therein). This observation motivates our terminology.
\end{remark}

The following result relates dynamical properties of group actions to the functional properties of group characters.

\begin{proposition}\label{PropositionCompressible} Let $G$ be a countable group admitting a compressible action
by homeomorphisms on some regular Hausdorff topological space $X$. Then for every non-regular indecomposable character $\chi$ of $G$ there exists $g\neq e$ such that $|\chi(g)|=1$.
\end{proposition}
\begin{proof}  Consider a proper indecomposable character $\chi$ of $G$. Assume that $|\chi(g)|<1$ for all $g\neq e$. Let $(\pi,H,\xi)$ be the GNS-construction associated to $\chi$.

 (1) We notice that  the definition of the compressible action implies that $\chi$ has the multiplicativity property in the sense that if $U_1,U_2\in \mathfrak{U}\,$ and $g,h\in G$ are such that
  \begin{equation*}\supp(g)\subset U_1,\supp(h)\subset U_2\;\text{and}\;  \overline{U}_1\cap\overline{U}_2=\varnothing\end{equation*} then \begin{equation}\label{mult}\chi(gh)=\chi(g)\chi(h).
  \end{equation} Indeed, find an increasing sequence of finite sets $F_n\subset G$  with $\bigcup_n F_n = G$. Then by the conditions (i) and (iv) of Definition \ref{DefinitionCompressibleBase}, we can find open sets $V_n\in \mathfrak U$ such that \[V_n\supset \bigcup\limits_{f\in F_n}\supp(f).\] By the condition (iii) there exist elements $r_n\in G$ such that \[r_n(U_1)\cap V_n=\varnothing\;\text{and}\;\supp(r_n)\cap U_2=\varnothing.\]
Then  $r_n h r_n^{-1} = h$  and  $supp(r_ng r_n^{-1})\cap supp(f) = \varnothing$ for every $f\in F_n$. Passing to a subsequence if needed, we can assume that $\pi(r_ng r_n^{-1})$  converges weakly to an operator $Q\in \mathcal M_\pi$. Notice that $tr(Q) = \chi(g)$. Since the operator $Q$ commutes with $\pi(F_n)$ for every $n$, we get that $Q$ belongs to the center of $\mathcal M_\pi$. Therefor, $Q$ is scalar and $Q = \chi(g) \id$. Thus $$\chi(g h) =
\lim_{n\to \infty} (\pi(r_ngh r_n^{-1})\xi,\xi) = (Q\pi(h)\xi,\xi) = \chi(g)\chi(h).$$

(2)   We claim that for any $\varepsilon>0$ and any open set $U$ there exists $g\in G$ with $supp(g)\subset U$ and $|\chi(g)| < \varepsilon$. Indeed, fix an element $h\neq e$ and $n\in \mathbb{N}$. Find $n$ subsets $V_1,\ldots,V_n\in\mathfrak{U}\,$ such that $\overline{V}_j\cap \overline{V}_k=\varnothing$ for $j\neq k$.  By assumptions (i) and (ii) we can choose elements $g_1,\ldots,g_n\in G$ such that $g_j(\supp (h))\subset V_j$ for each $j$. Set
 $$f=(g_1hg_1^{-1})(g_2hg_2^{-1})\cdots (g_nhg_n^{-1}).$$ By multiplicativity, we obtain that $\chi(f)=\chi(h)^n$. Choosing $n$ sufficiently large we get an element $f$ with $|\chi(f)|<\varepsilon$. By assumptions (i) and (ii) we can find an element $g$ conjugate to $f$ with $\supp(g)\subset U$, which proves the claim.

(3)  Consider an element $g\in G,\;g\neq e$. Find an open set $U$ with $\overline{g(U)}\cap \overline{U} = \varnothing$.  Fix $\varepsilon > 0$ and $n\in \mathbb{N}$. Using the condition (ii) and (iv) of Definition \ref{DefinitionCompressibleBase}, we can find  subsets $U_1,\ldots,U_n,V_1,\ldots,V_n\in \mathfrak{U}$ with pairwise disjoint closures such that $g(V_i)\subset U_i\subset U$ for each $i$.
Find $h_j\in G,j=1,\ldots, n$ supported by $U_j$ with $|\chi(h_j)| <\varepsilon$. Set $\xi_j = \pi(h_jg h_j^{-1})\xi$.  Then for $i\neq j$, the multiplicativity of $\chi$ implies that  \begin{eqnarray*}(\xi_{i},\xi_{j}) & = &  \chi(h_{j}g^{-1} h_{j}^{-1} h_{i}g h_{i}^{-1}) \\
& = & \chi(h_{j}(g^{-1} h_{j}^{-1}g)(g^{-1} h_{i}g)h_{i}^{-1}) \\
& = & \chi(h_{j})\chi(g^{-1} h_{j}^{-1}g)\chi(g^{-1} h_{i}g)\chi(h_{i}^{-1}).\end{eqnarray*}
As  $|\chi(h_j)| < \varepsilon$, we obtain that $|(\xi_{j},\xi_{i})| < \varepsilon$.
Thus,
\[\|\xi_1+\ldots+\xi_n\|\le \big(n+n(n-1)\varepsilon\big)^{\frac{1}{2}}.\]
Since $(\xi_l,\xi)=\chi(g)$ for
each $l$, we have $$|\chi(g)|=\tfrac{1}{n}|
(\xi_{1}+\xi_2+\ldots+\xi_{n},\xi)|\le
\tfrac{1}{n}\big(n+n(n-1)\varepsilon\big)^{\frac{1}{2}}.$$ When $n$
goes to infinity, we obtain
\[|\chi(g)|\le\varepsilon^{\frac{1}{2}}.\] Since $\varepsilon>0$ is
arbitrary, we conclude that $\chi(g)=0$. Thus, $\chi$ is the regular character.
\end{proof}

\begin{lemma}\label{LemmaIdentityCharacter} Let $G$ be a simple group and $\chi$ be a character.  If $|\chi(g)|=1$ for some $g\in G,g\neq e$, then
\[\chi(s)=1\;\;\text{for all}\;\;s\in G.\]
In particular, if $\chi$ is not the identity character, then
$|\chi(s)|<1$ for all $s\neq e$.
\end{lemma}
\begin{proof} Let $c=\chi(g)$, $|c|=1$. Consider the GNS construction $(\pi,H,\xi)$ corresponding to $\chi$. Using the Cauchy-Schwarz inequality
and the fact that the vector $\xi$ is separating, we obtain that
\[(\pi(g)\xi,\xi)=c\;\;\Rightarrow\;\;\pi(g)\xi=c\xi\;\;\Rightarrow
\pi(g)=c\id.\] Take an arbitrary element $h\in G$ which does not commute with $g$ and set
$s=hgh^{-1}g^{-1}$. Then $\pi(s)=\id$. It follows that
$\pi(s_1)=\id$ for all $s_1$ from the normal subgroup generated by
$s$. Since $G$ is simple, we get that $\pi(g) = \id$ for every $g\in G$. Thus, $\chi$ is the identity character.
\end{proof}

As a corollary of Lemma \ref{LemmaIdentityCharacter} and Proposition \ref{PropositionCompressible}
we immediately obtain the following result.

\begin{theorem}\label{TheoremNoCharactersSimpleGroups}  Let $G$ be a simple  countable group admitting a compressible action on a regular Hausdorff topological space $X$. Then $G$ has no  proper characters.
\end{theorem}

Let $G$ be a group. For a subgroup $R$ of $G$ and an element $g\in G$ set $C_R(g)=\{hgh^{-1}:h\in R\}$. Denote by $N(R)$ the normal closure of $R$ in $G$, i.e., the subgroup of $G$ generated by all elements of the form $grg^{-1},g\in G, r\in R$.
\begin{theorem}\label{TheoremCharactersNonSimpleGroups}  Let $G$ be a group and  $R$ be an ICC  subgroup of $G$ such that
\begin{itemize}

\item[(i)] $R$ has no proper characters;

\item[(ii)] for every $g\in G\setminus\{e\}$, there exists a sequence of distinct elements $\{g_i\}_{i\geq 1}\subset C_R(g)$ such that $g_i^{-1} g_j\in R$ for any $i,j$.
\end{itemize}
 Then each finite type factor representation $\pi$ of $G$ is either  regular  or  has the
form $$\pi(g)=\omega([g]),$$ where $\omega$ is a finite factor representation of $G/N(R)$ and $[g]\in G/N(R)$ is the coset of the
element $g$.
\end{theorem}
\begin{proof}  Consider an indecomposable character $\chi$ of $G$. Let $(\pi,H,\xi)$ be the GNS-construction associated to $\chi$.

(1) Consider the restriction of $\pi$ onto the subgroup $R$. Set $H_{_R}=\overline{Lin(\pi(R)\xi)}$. Since the restriction of $\chi$ on $R$ is a character and the only indecomposable characters of the group $R$ are the regular and the identity characters, we can decompose the space $H_{_R}$ into  $R$-invariant subspaces $H_1$ and $H_2$ (possibly trivial) such that $H_{_R} = H_1 \bigoplus H_2$ with  $\pi(R)|H_1$ being the identity representation and
 $\pi(R)| H_2$ being the  regular representation.

  The orthogonal projections $\{P_i\}$ onto $H_i$, $i=1,2$ belong to the center of the algebra generated by $\pi(R)$. In particular, $P_i$ lies in the algebra $\mathcal{M}_\pi$. Furthermore, \[\chi(g)=\alpha\chi_{id}(g)+(1-\alpha)\chi_{reg}(g)\;\;\text{for
all}\;\;g\in R,\] where $\chi_{id}$ is the identity character,  $\chi_{reg}$ is the
regular character, and $\alpha\in [0,1]$. If $\alpha\neq 0,1$, we can write down the vector $\xi$ as
\begin{equation}\label{EqXi}\xi=\alpha^{\frac{1}{2}}\xi_1+(1-\alpha)^{\frac{1}{2}}\xi_2,\end{equation}
 where $\xi_1\in H_1$, $\xi_2\in H_2$ are unit vectors such that \[(\pi(h)\xi_1,\xi_1)=
 \chi_{id}(h)=1,\;\;(\pi(h)\xi_2,\xi_2)=
 \chi_{reg}(h)=\delta_{h,e}\;\;\text{for all}\;\;h\in R.\] For convenience, if $\alpha=0$, we set $\xi_1=0,\xi_2=\xi$, if $\alpha=1$, we set $\xi_1=\xi,\xi_2=0$.
Observe that $H_i = \overline{Lin(\pi(R)\xi_i)}$, $i=1,2$.

(2) Assume that $H_2\neq\{0\}$. Consider an arbitrary element $g\in G$, $g\neq e$. By our assumptions there exists a sequence of  elements $\{h_n\}\in R\setminus \{e\}$ such that $h_m^{-1} g^{-1}h_mh_n^{-1}gh_n\in R$ for all $n$ and $m$ and elements $h_n^{-1}gh_n$ are pairwise distinct. Set $g_m = h_m^{-1} gh_m$. Since $g_m^{-1}g_n \in R\setminus\{e\}$, we get that $$(\pi(g_n)\xi_2,\pi(g_m)\xi_2) =  \chi_{reg}(g_m^{-1}g_n) = 0.$$
This shows that $\pi(g_m)\xi_2\to 0$ weakly. Observe also that \begin{eqnarray*}(\pi(g_n)\xi_2,\xi_2) &=& (\pi(g_n)P_2\xi,P_2\xi)  = tr(P_2 \pi(h_n^{-1} g h_n)P_2)\\
& =& tr(\pi(h_n^{-1}) P_2 \pi(g) P_2 \pi(h_n)) = tr(P_2\pi(g)P_2).\end{eqnarray*} Since the latter is independent of $n$ and $\pi(g_n)\xi_2\to 0$, we conclude that $$(\pi(g)\xi_2,\xi_2) = tr(P_2\pi(g)P_2) = 0.$$
 Set $H_0 = \overline{Lin(\pi(G)\xi_2)}$.  Then $\pi(G)|H_0$ is quasi-equivalent to the regular representation.  Since $\pi$ is a factor representation, we conclude that $\pi$ is the regular representation.

(3) Assume that $H_2 = \{0\}$. Then $\xi=\xi_1$ and $\pi(h)=\id$ for every $h\in R$. Therefor, $\pi(g)=\id$ for all $g\in N(R)$. This means that the representation $\pi$ factors through the quotient $G/N(R)$ and defines a finite type factor representation $\omega$ of $G/N(R)$ such that
$\pi(g)=\omega([g])$ for all $g\in G$.
\end{proof}

 Recall that a finite factor representation of a group $G$ is of {\it type $I$} if the von Neumann algebra of the representation is isomorphic to the algebra of all linear operators in some finite-dimensional Hilbert space. We  say that an action of group $G$  on a measure space $(Y,\mu)$ is {\it trivial} if $g(x)=x$ for every $g\in G$ and $\mu$-almost every $x\in X$.  The following result shows that any ergodic action of a group admitting no  characters must be {\it essentially free}, that is $\mu(Fix(g))=0$ for all $g\in G\setminus\{e\}$.

\begin{theorem}\label{TheoremErgodicAction} Assume that every finite factor representation of a countable ICC group $G$ is either regular or of type $I$ and that there is at most a countable number (up to quasi-equivalence) of finite factor representations of $G$. Then every faithful ergodic measure-preserving action of $G$ is essentially free.
\end{theorem}
\begin{proof} Consider an ergodic action of $G$ on a measure space $(Y,\mu)$.  Set $$\widetilde{Y}=\{(x,y)\in Y\times Y | x=g(y)\;\text{for some}\;g\in G\}.$$ For a Borel set $A\subset \widetilde Y$ and a point $x\in Y$, set $A_x =\{(x,y)\in A\}$.  Define a $\sigma$-finite measure $\widetilde \mu$ on $\widetilde Y$ by $\widetilde \mu(A) = \int_Y card(A_x)d\mu(x)$. Given a function $f\in L^2(\widetilde Y,\widetilde \mu)$ and a group element $g\in G$, set $$(\pi(g)f)(x,y)=f(g^{-1}x,y).$$

Then $\pi(g)$ is a unitary operator on the Hilbert space $L^2(\widetilde Y,\widetilde \mu)$.
Denote by $\xi$ the indicator function of the diagonal of $Y\times Y$. Set $H=\overline{Lin\{\pi(G)\xi\}}$. We note the  von Neumann algebra  $\mathcal{M}_\pi$  generated by $\pi(G)$, restricted to $H$, is of  finite type. We refer the reader to \cite{Feldman_Moore:I} for the details.
Since the group $G$ has at most a countable number of finite factor representations, the representation $\pi$ decomposes into a direct sum (at most countable) of factor representations.

Our goal is to show that the representation $\pi$ is regular. Then the uniqueness of the trace implies that  $(\pi(g)\xi,\xi)=0$ for every $g\neq e$. Using the identity $\mu(Fix(g)) = (\pi(g)\xi,\xi)$, we get that the action is essentially free.

 Suppose to the contrary that  the decomposition of $\pi$ into factors contains a {\it non-regular} factor representation $\pi_1$, which, by our assumptions,  generates  a finite-dimensional  von Neumann factor.
Let $P_1$ be a projection from the center of $\mathcal M_\pi$ such that $\pi_1(g)=P_1\pi(g)$ for every $g\in G$. Set $\xi_1=P_1\xi$.

Since for every $g\in G$ the unitary operator $(\pi'(g)f)(x,y)=f(x,g^{-1}y)$ belongs to $\mathcal M_\pi'$ and $\pi'(g)\xi=\pi(g^{-1})\xi$, we have that $$\pi'(g)\pi(g)\xi_1 = \pi'(g)\pi(g)P_1\xi = P_1\pi'(g)\pi(g)\xi =  P_1\xi = \xi_1$$ for all $g\in G$.  This implies that the function  $h(x):=|\xi_1(x,x)|$ is $G$-invariant and $\mu$-integrable on $Y$. By the ergodicity, we get that $h(x) \equiv C$ on $Y$ for some constant $C$.  Note that if $C=0$, then  $$0 = \int_{\widetilde Y} \xi_1(x,y)\xi(x,y)d\widetilde \mu(x,y) = (\xi_1,\xi),$$ which is impossible as the projection of $\xi$ onto $\xi_1$ is non-trivial.

 Fix an orthonormal basis $\eta_1,\ldots,\eta_n$ for  $\overline{Lin\{\pi_1(G)\xi_1\}}$. For a given $g\in G$, write
$$\pi_1(g)\xi_1=\sum\limits_{j=1}^n\alpha_j(g)\eta_j$$ for some $\alpha_1(g),\ldots,\alpha_n(g)$ with $\sum|\alpha_j(g)|^2=|\xi_1|^2\le 1$. Observe that $$\sum\limits_{j=1}^n\alpha_j(g)\eta_j(x,y)=(\pi_1(g)\xi_1)(x,y)=
(\pi(g)\xi_1)(x,y)=\xi_1(g^{-1}x,y)$$ for every $(x,y)\in \widetilde Y$.  Since $|\xi_1(g^{-1}x,y)| = C$ for $(x,y)\in \widetilde Y$ with $x= g(y)$, we conclude that $\sum\limits_{j=1}^n|\eta_j(x,y)|\ge C >0$ for $(x,y)$ with $x = gy$ and, thus, for any $(x,y)\in\widetilde Y$. This implies that  the function $\sum\limits_{j=1}^n|\eta_j(x,y)|$ is not integrable with respect to $\widetilde \mu$. This contradiction yields that $\pi_1 = 0$ and, thus,  the representation $\pi$  is regular. \end{proof}

We finish this section by giving examples of groups admitting no compressible actions. We observe that even though the following proposition yields a  result similar to that of  Theorem \ref{TheoremCharactersNonSimpleGroups}, the underlying assumptions are different and not mutually interchangeable.

\begin{proposition} Let $G$ be a  countable group with trivial center and such that every proper quotient  is  finite or abelian. Assume that the group $G$ admits a compressible action on a regular Hausdorff space. Then all finite (Murray von Neumann) non-regular representations of $G$ are of type $I$.
\end{proposition}
\begin{proof} Consider a non-regular indecomposable character $\chi$ of $G$. Let $(\pi,H,\xi)$ be the GNS-construction associated to $\chi$. By Proposition \ref{PropositionCompressible} there exists $g\neq e$ such that $|\chi(g)|=1$. Choose $h\in G$  not commuting with $g$. Denote by $N$ the normal subgroup of $G$ generated by the element $ghg^{-1}h^{-1}$. Using the arguments from the proof of Lemma \ref{LemmaIdentityCharacter} we obtain that $\pi|N=\id$. Thus, the representation $\pi$ of the group $G$ gives rise to the representation of $G/N$ with the same von Neumann algebra. Recall that factor representations of abelian groups are scalar.
\end{proof}

If a group $G$ as in the proposition above has a measure-preserving  action on a measures space $(X,\mu)$  with $0<\mu(Fix(g))<1$ for some $g\neq e$, then, in view of Theorem \ref{TheoremErgodicAction}, such a group cannot have compressible actions. Examples of such groups are  full groups of even Bratteli diagrams, commutators of topological full groups of Cantor minimal  systems \cite{Dudko_Medynets:2012}, and just infinite branch groups   \cite{Grigorchuk:2011}.


\section{Applications}\label{Applications}

In this section we show that the results established  in the previous section  are applicable to the Hignam-Thompson groups and to the full groups of irreducible shifts of finite type.

\subsection{The Higman-Thompson groups}

\begin{definition} Fix two positive integers $n$ and $r$. Consider an interval $I_r = [0,r]$. Define the group $F_{n,r}$ as the set of all orientation preserving piecewise linear homeomorphisms $h$ of $I_r$ such that all singularities of $h$ are in $\mathbb Z[1/n]=\{\tfrac{p}{n^k}:p,k\in\mathbb N\}$;
the derivative of $h$ at any non-singular point is $n^k$ for some $k\in \mathbb Z$.

\end{definition}

Observe that the commutator subgroup of $F_{n,r}$ is a simple group and the abelianization of $F_{n,r}$ is isomorphic to $\mathbb Z^n$ \cite[Section 4]{Brown:1987}.   Consider the subgroup $F_{n,r}^0$ of $F_{n,r}$ consisting of all elements $f\in F_{n,r}$ with $ supp(f)$ being  a  subset of $(0,r)$. Observe that (the commutator subgroup) $F_{n,r}' = (F_{n,r}^0)'$ \cite[Section 4]{Brown:1987}.
The following lemma shows that the commutator  subgroup $F_{n,r}'$ satisfies the assumptions of Theorem \ref{TheoremNoCharactersSimpleGroups}.

\begin{lemma}\label{LemmaRUConditions} The base of topology $\mathfrak{U}\,=\{(a,b):[a,b]\subset (0,r),a,b\in\mathbb Z[\tfrac{1}{n}]\}$ satisfies the conditions (i)-(iv) of Definition \ref{DefinitionCompressible} for the action of the group $R = (F_{n,r}^0)'$ on $(0,r)$. Thus, the action of $R$ is compressible.
\end{lemma}
\begin{proof}
The conditions (i) and (iv) of Definition \ref{DefinitionCompressible} are clearly satisfied.

To check the condition (ii), consider intervals $U_1=(a,b)$ and $U_2=(c,d)$ both in $\mathfrak{U}$. Replacing $U_2$ by a subinterval if necessary we may assume that $\tfrac{b-a}{d-c}=n^k$ for some $k\in\mathbb Z$. Since the function $\tfrac{a-x}{c-x}$ is continuous in $x$ for $x\neq c$, we can find $x\in \mathbb Z[\tfrac{1}{n}]$ such that $0<x<\min\{a,c\}$ and $\tfrac{a-x}{c-x}=n^k$ for some $k\in \mathbb Z$.  Similarly, we can find $y\in\mathbb Z[\tfrac{1}{n}]$, $\max\{b,d\}<y<r$  such that $\tfrac{y-b}{y-d}=n^k$ for some $k\in\mathbb Z$. Let $g:[0,r]\to [0,r]$ be the function such that
\[g(0)=0,\;g(x)=x,\;g(a)=c,\;g(b)=d,\;g(y)=y,\;g(r)=r,\] and $g$ is linear on each of the line segments $[0,x],[x,a],[a,b],[b,y],[y,r]$. Then $g\in R$ and $g(U_1)=U_2$.

To check the condition (iii), we consider intervals  $U_i=(a_i,b_i)$, $i=1,2,3$ from $\mathfrak U$ such that $\overline{U}_1\cap \overline{U}_2= \emptyset$. Without loss of generality, assume that $a_1>b_2$. Set $a=a_1,b=b_1$ and fix $c,d\in\mathbb Z[\tfrac{1}{n}]$ with $\max\{b_2,b_3\}<c<d<r$ and $\tfrac{b-a}{d-c}=n^k$ for some $k\in\mathbb Z$. Construct $g$ as above with $a_1>x>b_2$ and $r>y>d$. Then $\supp(g)\subset [x,y]$ and $g(U_1)=(c,d)$. Therefore, $g(U_1)\cap U_3=\varnothing$ and  $\supp(g)\cap U_2=\varnothing$.
  \end{proof}

Observe that all finite factor representations of abelian groups are scalar representations, i.e. $\pi(g) = c_g\id$, with $c_g\in \mathbb T$, the unit circle.  In particular, the indecomposable characters of abelian groups are homomorphisms into  $\mathbb T$.

\begin{corollary}\label{CorollaryFnr} (1) The group $F_{n,r}'$ has no proper characters.
 (2) If $\chi$ is an indecomposable character of $F_{n,r}$, then $\chi$ is either regular or $\chi(g) = \rho([g])$, where $[g]$ is the image of $g$ in the abelianization of $F_{n,r}$ and $\rho: \mathbb Z^n \rightarrow \mathbb T$ is a group homomorphism.
\end{corollary}
\begin{proof} Statement (1) immediately follows from Lemma \ref{LemmaRUConditions} and Theorem  \ref{TheoremNoCharactersSimpleGroups}.

 To establish the second result, we only need to check the condition (2) of  Theorem \ref{TheoremCharactersNonSimpleGroups}. Fix $g\in F_{n,r}\setminus \{e\}$. Find an interval $(a,b)$ with $g(a,b)\cap (a,b)= \emptyset$. Find a sequence of distinct elements $\{h_n\}_{n\geq 1}\subset (F_{n,r})'$ supported by $(a,b)$. Then $(h_n^{-1} g^{-1} h_n )(h_m^{-1}g h_m)\in (F_{n,r})'$ for any $n\neq m$. This completes the proof.
\end{proof}


\begin{definition} Let $n$ and $r$ be positive integers. Define Higman's group $G_{n,r}$ as the group of all right continuous bijections of $[0,r)$ which are piecewise linear, with finitely many discontinuities  and singularities, all in $\mathbb Z[1/n]$, slopes in $\{n^k : k\in\mathbb Z\}$, and mapping $\mathbb Z[1/n]\cap [0,r)$   to itself.
\end{definition}

Note that $F_{n,r}\subset G_{n,r}$. In fact, $F_{n,r}$ consists exactly of all continuous elements $g\in G_{n,r}$. In \cite{Higman:1974} Higman showed that the commutator subgroup $G_{n,r}\rq{}$ is simple and that the abelianization of $G_{n,r}$ is trivial for even $n$ and is $\mathbb Z/2\mathbb Z$ for odd $n$.

\begin{lemma}\label{LemmaConditionsOnGroups} The groups $R = F_{n,r}'$ and $G = G_{n,r}$ satisfy the conditions of Theorem \ref{TheoremCharactersNonSimpleGroups}.
\end{lemma}
\begin{proof} Corollary \ref{CorollaryFnr} shows that the group $R$ has no proper characters. Consider an arbitrary element $g\in G,g\neq e$. Choose an open interval $I$ such that $I\cap g^{-1}(I)=0$ and $g$ is continuous on both $I$ and $g^{-1}(I)$. It follows that for any two elements $r_1,r_2\in R$ with $\supp(r_1)\subset I$, $\supp(r_2)\subset I$ the element $$h=r_2g^{-1}r_2^{-1}r_1gr_1^{-1}\neq e$$ is a continuous bijection of $[0,r)$. Observe that $h$ acts identically near $0$ and $r$.
It follows that $h\in R$ and the elements $r_1gr_1^{-1}$  and $r_2gr_2^{-1}$ belong to the same coset of $G/R$. Since the group $R$ has infinitely many elements supported by the set $I$, we immediately establish  the condition  (ii).
\end{proof}

The following result is an immediate corollary of  Theorem \ref{TheoremCharactersNonSimpleGroups} applied twice to the pairs  $R = F_{n,r}'$, $G = (G_{n,r})'$ and  $R = F_{n,r}'$, $G = G_{n,r}$.

\begin{corollary}  (1) The group $G_{n,r}'$ has no proper characters.

(2) If $\chi$ is an indecomposable character of $G_{n,r}$, then $\chi$ is either regular or $\chi(g) = \rho([g])$, where $[g]$ is the image of $g$ in the abeliazation of $G_{n,r}$ and $\rho: G_{n,r}/G_{n,r}' \rightarrow \mathbb T$ is a group homomorphism.
\end{corollary}

\subsection{Full groups of irreducible shifts of finite type}

We refer the reader to \cite[Section 6]{Matui:2012} for the comprehensive study of full groups of \'etale groupoids including the groups discussed below.

Let $(V,E)$ be a finite directed graph. Suppose that the adjacency matrix of the graph is irreducible and is not a permutation matrix. For an edge $e\in E$, denote by $i(e)$ the initial vertex and by $t(e)$ its terminal vertex.  Set $$X = \{\{e_n\}_{n\geq 1}\in E^\mathbb N : t(e_k) = i(e_{k+1})\mbox{ for every }k\in \mathbb X\}.$$
Equipped with the product topology,  $X$ is a Cantor set.  We note that the space $X$ along with the left shift is called a one-sided subshift of finite type, see \cite{Matui:2012} and references therein regarding relations with the symbolic dynamics.

An $n$-tuple $(e_1,\ldots,e_n)\in E^n$  is called {\it admissible} if $t(e_k) = i(e_{k+1})$ for every $1\leq k\leq n-1$. Two admissible tuples $\overline e = (e_1,\ldots,e_n)$ and $\overline f = (f_1,\ldots,f_m)$ are called {\it compatible} if $t(e_n) = t(f_m)$. Each admissible tuple $\overline e = (e_1,\ldots,e_n)$ defines a clopen set $U(\overline e) = \{x\in X : x_i = e_i,\;i=1,\ldots,n\}$. Such clopen sets form the base of topology.  Given two compatible admissible tuples $\overline e_1$  and $\overline e_2$, define a continuous map $\pi_{\overline e_1,\overline e_2}: U(\overline e_1)\rightarrow U(\overline e_2)$ as
$$\pi_{\overline e_1,\overline e_2}(\overline e_1, x_{n+1},x_{n+2},\ldots) = (\overline e_2,x_{n+1},x_{n+2},\ldots).$$

\begin{definition} Following \cite{Matui:2012}, we define the {\it full group} of $X$, in symbols $[[X]]$, as the set of all homeomorphisms $g$ of $X$ for which there exists two clopen partitions $X = \bigsqcup_{i=1}^n U(\overline e_i) = \bigsqcup_{i=1}^n U(\overline f_i)$ with $e_i$ and $f_i$ being compatible admissible tuples (possibly of different lengths), $i=1,\ldots, n$, such that $g|U_{\overline e_i} = \pi_{\overline e_i, \overline f_i}$ for every $i=1,\ldots, n$.

For a clopen subset $Y\subset X$, set $[[X|Y]]$ as the set of all $g\in [[X]]$ with $\supp(g)\subset Y$.
\end{definition}
The following result was established in \cite[Lemma 6.1 and Theorem 4.16]{Matui:2012}

\begin{proposition}\label{PropositionSimpleFullGroup} For any clopen set $Y\subset X$,  the commutator group $[[X|Y]]'$ is simple.
\end{proposition}

Fix an arbitrary point $x_0\in X$. Find an increasing sequence of clopen sets $\{Y_n\}$  such that $X\setminus \{x_0\} = \bigcup_nY_n$. Set $R = \bigcup_n [[X|Y_n]]'$. It follows from Proposition \ref{PropositionSimpleFullGroup} that the group $R$ is simple. Observe that the group $R$ consists of all elements $g\in [[X]]'$ equal to the identity on some neighbourhood of $x_0$.

Denote by $\mathcal F$ the set of all admissible tuples which are {\it not prefixes} of $x_0$. Define $\mathfrak{U}$ as the family of all finite unions of sets from  $\{U(\overline e)\}_{\overline e \in \mathcal F}$. Notice that $\mathfrak{U}$ is a base of the topology on $X\setminus \{x_0\}$. One can check that $\mathfrak U$ satisfies conditions (i)-(iv) of Definition \ref{DefinitionCompressibleBase} for the action of $R$.  Thus, using Theorem \ref{TheoremNoCharactersSimpleGroups}, we conclude that the group $R$ has no  characters.  Considering $R$ as a subgroup of $G = [[X]]$, one can check that the assumptions of Theorem \ref{TheoremCharactersNonSimpleGroups} are satisfied. We leave the details to the reader.

\begin{corollary} If $\chi$ is an indecomposable character of $[[X]]$, then $\chi$ is either regular or $\chi(g) = \rho([g])$, where $[g]$ is the image of $g$ in the abelianization of $[[X]]$ and $\rho: [[X]]/[[X]]' \rightarrow \mathbb T$ is a group homomorphism.
\end{corollary}

To finish our discussion, we notice that  the full group of the one-sided Bernoulli shift over the alphabet with $n$ letters is isomorphic to $G_{n,1}$ \cite{Nekrashevych:2004}. \medskip

{\bf Acknowledgements. } We would like to thank R.~Grigorchuk for the discussion of  Higman-Thompson groups and for his valuable comments.

\end{document}